\documentclass[envcountsame,envcountsect,runningheads,a4paper]{llncs}

\usepackage[all, knot]{xy}

\usepackage{amsmath}
\usepackage{amssymb}
\usepackage{amsfonts}
\usepackage{amsxtra}
\usepackage{graphicx}

\usepackage{url}
\newcommand{\keywords}[1]{\par\addvspace\baselineskip
\noindent\keywordname\enspace\ignorespaces#1}

\spnewtheorem{construction}[theorem]{Construction}
{\bfseries}{}
\spnewtheorem{myexample}[theorem]{Example}
{\bfseries}{}
\spnewtheorem{myremark}[theorem]{Remark}
{\bfseries}{}
\spnewtheorem{mydefinition}[theorem]{Definition}
{\bfseries}{}

\begin{document}

\mainmatter  

\title{Commuting Groups and the Topos of Triads}

\titlerunning{Commuting Groups and the Topos of Triads}

\author{Thomas M. Fiore and Thomas Noll}


\institute{Department of Mathematics and Statistics\\
University of Michigan-Dearborn \\ 4901 Evergreen Road \\ Dearborn,
MI 48128, U.S.A. \\ \texttt{tmfiore@umd.umich.edu}  \\  Escola Superior de M\'{u}sica de Catalunya \\
Departament de Teoria, Composici\'{o} i Direcci\'{o} \\
C. Padilla, 155 - Edifici L'Auditori \\
08013 Barcelona, Spain \\ \texttt{noll@cs.tu-berlin.de}  }

\toctitle{Commuting Groups and the Topos of Triads}
\tocauthor{Thomas M. Fiore and Thomas Noll}
\maketitle

\begin{abstract}
The goal of this article is to clarify the relationship between the topos of triads and the neo-Riemannian $PLR$-group. To do this, we first develop some theory of generalized interval systems: 1) we prove the well known fact that every pair of dual groups is isomorphic to the left and right regular representations of some group (Cayley's Theorem), 2) given a simply transitive group action, we show how to construct the dual group, and 3) given two dual groups, we show how to easily construct sub dual groups. Examples of this construction of sub dual groups include Cohn's hexatonic systems, as well as the octatonic systems. We then enumerate all $\mathbb{Z}_{12}$-subsets which are invariant under the triadic monoid and admit a simply transitive $PLR$-subgroup action on their maximal triadic covers. As a corollary, we realize all four hexatonic systems and all three octatonic systems as Lawvere--Tierney upgrades of consonant triads.
\keywords{$PLR$-group, duality, sub dual groups, hexatonic systems, octatonic, topos, topos of triads, triadic monoid}
\end{abstract}

\section{Introduction}
In \cite{cohn1997}, Richard Cohn coined the term ``the overdetermined triad" in order to draw the reader's attention away from the celebrated acoustic attractiveness of the major and minor triads to some surprising algebraic properties of the family of the consonant triads. The starting point for the present investigation is a puzzling ramification of this over-determination of the triads even within the realm of algebraic music theory. In particular, the musical and music-theoretical prevalence of three tone-collections --- major/minor mixture, hexatonic, and octatonic --- emerges from the study of triads in quite different algebraic ways.

In the context of a group-theoretic approach (see \cite{fioresatyendra2005}) one obtains these collections as carrier sets, which are covered by particular families of triads, which are orbits under certain group actions. Under this perspective it is the interplay of several triads, which makes these collections prevalent among others.

In the context of a monoid-theoretic approach (see \cite{nollToposOfTriads}) one obtains these collections as extensions of a single triad. The stability properties of the set $\{0, 4, 7\}$ within $\mathbb{Z}_{12}$ can be studied in terms of a monoid of 8 affine transformations. On top of this monoid one obtains three interesting Lawvere--Tierney topologies, which also imply the prevalence of the three sets $\{0, 3, 4, 7\}$,  $\{0, 3, 4, 7, 8, 11\}$ and $\{0, 1, 3, 4, 6, 7, 9, 10\}$ as extensions of a single triad.

How do these two approaches fit together? The present article prepares the group-theoretical grounding for a comprehensive answer and anticipates the remaining work on the topos-theoretical side in the concrete case of $\mathbb{Z}_{12}$.

Though this article does not require any topos theory, we have included an Appendix on Topos Theory in Section~\ref{sec:appendix} for any readers who want to recall the main notions.

\section{Dual Groups and Cayley's Theorem}

We first briefly recall the notion of dual groups in transformational analysis, give the fundamental example of left and right regular representations, and show how to construct general dual groups in view of the left and right regular representations. As an example, we reconstruct the well-known $PLR$-group as the dual group to the $T/I$-group acting on consonant triads.

\begin{mydefinition}[Dual groups in the sense of Lewin]
Let $\text{\rm Sym}(S)$ be the symmetric group on the set $S$. Two subgroups $G$ and $H$ of the symmetric group $\text{\rm Sym}(S)$ are called {\it dual} if their natural actions on $S$ are simply transitive and each is the centralizer of the other, that is,
$$C_{\text{\rm Sym}(S)}(G)=H \;\; \text{ and } \;\; C_{\text{\rm Sym}(S)}(H)=G.$$
\end{mydefinition}

\begin{myexample}[Left and right regular representations] \label{examp:left/right_regular_representations}
If $G$ is any group, the left and right regular representations $\lambda,\rho \colon G \to \text{Sym}(G)$ are defined by $\lambda_g(h)=gh$ and $\rho_g(h)=hg^{-1}$. As is well known, $\lambda$ and $\rho$ are injective and the natural actions of $\lambda(G)$ and $\rho(G)$ are simply transitive. The images $\lambda(G)$ and $\rho(G)$ clearly commute, since $\lambda_{g_1}\rho_{g_2}(h)=g_1hg_2^{-1}=\rho_{g_2}\lambda_{g_1}(h)$. In fact, $\rho(G)$ is the centralizer of $\lambda(G)$ in $\text{Sym}(G)$, for if $f\colon G \to G$ is a bijection satisfying $f(gh)=gf(h)$ for all $g$ and $h$ in $G$, then $f(h)=f(he)=hf(e)$ and $f$ is the same as $\rho(f(e)^{-1})$. Similarly, $\lambda(G)$ is the centralizer of $\rho(G)$ in $\text{Sym}(G)$. The groups $\lambda(G)$ and $\rho(G)$ provide the most basic example of dual groups. All other examples of dual groups are isomorphic to this one, as we see in the following construction.
\end{myexample}

\begin{construction}[Construction of the dual group] \label{construction:dual_group}
Given a group $G$ acting simply transitively on a set $S$, we may construct the dual group as follows.
We consider $G$ as a subgroup of $\text{Sym}(S)$ via the homomorphism $G \to \text{Sym}(S), g \mapsto (s \mapsto gs)$, which is an injection by simple transitivity. We suggestively call this subgroup $\lambda(G)$, although it is not the same as $\lambda(G)$ in Example~\ref{examp:left/right_regular_representations}. Fix an element $s_0 \in S$. Again by simple transitivity, the function $G \to S, h \mapsto hs_0$ is a bijection. The natural action of $\lambda(G)$ on $S$ is essentially the same as left multiplication: $g(hs_0)=(gh)s_0$. We now define a second injection $G \to \text{Sym}(S)$ by $g \mapsto (hs_0 \mapsto hg^{-1}s_0)$ and call the image $\rho(G)$. The group $\rho(G)$ is the dual group to $\lambda(G)$, which we sought to construct. The bijection $G \to S, h \mapsto hs_0$ and the injection $G \to \text{Sym}(S), g \mapsto (hs_0 \mapsto hg^{-1}s_0)$ depend on the choice of $s_0$, but the group $\rho(G)$ does not.
\end{construction}

\begin{myexample}[Construction of $PLR$-group]
If in Construction~\ref{construction:dual_group} we take $G$ to be the $T/I$-group, $S$ to be the set of major and minor triads, and $s_0=\{0,4,7\}$, then $\rho(G)$ is the neo-Riemannian $PLR$-group. For example, the operations parallel, leading tone exchange, and relative, denoted $P$, $L$, and $R$ respectively, correspond to right multiplication by $I_7$, $I_{11}$, and $I_4$.
\[ \aligned
P: \;\; & T_n\{0,4,7\} \mapsto T_nI_7\{0,4,7\} & \text{ and }  \;\; & I_n\{0,4,7\} \mapsto I_nI_7\{0,4,7\} \\
L: \;\; & T_n\{0,4,7\} \mapsto T_nI_{11}\{0,4,7\} & \text{ and }  \;\; & I_n\{0,4,7\} \mapsto I_nI_{11}\{0,4,7\} \\
R: \;\; & T_n\{0,4,7\} \mapsto T_nI_4\{0,4,7\} & \text{ and }  \;\; & I_n\{0,4,7\} \mapsto I_nI_4\{0,4,7\}.
\endaligned
\]
The embedding $\rho\colon T/I \to \text{Sym}(S)$ takes the generators $I_1$ and $I_8$ to $L$ and $R$ respectively (recall $I_1 \circ I_8=T_{1-8}=T_5$). If we choose a different $s_0$, then we obtain another isomorphism of the $T/I$-group with the $PLR$-group.
\end{myexample}

\section{Sub Dual Groups}

We next give a very practical method for constructing sub dual groups from dual groups. This method minimizes computation: instead of checking commutativity of all functions on every element, it suffices to evaluate a putative commuting function on the single element $s_0$ and check if the output lies in $S_0$.

\begin{theorem}[Construction of sub dual groups] \label{thm:sub_dual_groups}
Let $G,H \leq \text{\rm Sym}(S)$ be dual groups, $G_0$ a subgroup of $G$, and $s_0$ an element in $S$. Let $S_0$ be the orbit of $s_0$ under the action of $G_0$. Then the following hold.
\begin{enumerate}
\item \label{thm:sub_dual_groups:G_0_simple_transitivity}
The group $G_0$ acts simply transitively on $S_0$.
\item \label{thm:sub_dual_groups:s_0_determines_G_0}
If $g \in G$ and $gs_0$ is in $S_0$, then $g \in G_0$. In particular, if $g \in G$ and $gs_0$ is in $S_0$, then $g$ preserves $S_0$ as a set, that is, $g(S_0) \subseteq S_0$.
\item \label{thm:sub_dual_groups:H_0_simple_transitivity}
Let $H_0$ denote the subgroup of $H$ consisting of those elements $h\in H$ with $hs_0 \in S_0$.  Then $H_0$ acts simply transitively on $S_0$.
\item
Restriction from $S$ to $S_0$ embeds $G_0$ and $H_0$ in $\text{\rm Sym}(S_0)$. We denote their images in $\text{\rm Sym}(S_0)$ by $G_0\vert_{S_0}$ and $H_0\vert_{S_0}$.
\item \label{thm:sub_dual_groups:extensions}
If $g \in \text{\rm Sym}(S_0)$ and $g$ commutes with $H_0\vert_{S_0}$, then $g$ admits a unique extension to $S$ which belongs to $G$. This extension necessarily belongs to $G_0$ by \ref{thm:sub_dual_groups:s_0_determines_G_0}. Similarly, if $h \in \text{\rm Sym}(S_0)$ commutes with $G_0\vert_{S_0}$, then $h$ admits a unique extension to $S$ which belongs to $H$. This extension necessarily belongs to $H_0$.
\item \label{thm:sub_dual_groups:restricted_subgroups_dual}
The groups $G_0\vert_{S_0}$ and $H_0\vert_{S_0}$ are dual in $\text{\rm Sym}(S_0)$.
\end{enumerate}
\end{theorem}
\begin{proof}
\begin{enumerate}
\item
By the definition $S_0=G_0s_0$, we have $gs \in S_0$ for all $g \in G_0$ and $s \in S_0$. The orbit of $s_0$ is all of $S_0$, so the action of $G_0$ on $S_0$ is transitive. Simplicity of the $G_0$-action on $S_0$ follows from the simplicity of the $G$-action on $S$.
\item
Suppose $g \in G$ and $gs_0=s \in S_0$. By the simple transitivity of $G_0$ acting on $S_0$ from \ref{thm:sub_dual_groups:G_0_simple_transitivity}, there is a unique $\overline{g} \in G_0$ such that $\overline{g} s_0=s$. But on the other hand, $g$ and $\overline{g}$ are in $G$, which acts simply transitively on $S$, so $g=\overline{g}\in G_0$.
\item
Suppose $h \in H$ and $hs_0 \in S_0=G_0s_0$. An arbitrary element of $S_0$ has the form $gs_0$ for some $g \in G_0$, and $h(gs_0)=g(hs_0)\in G_0s_0=S_0$. Thus $h(S_0) \subseteq S_0$, and the action of $H_0$ on $S$ restricts to an action on $S_0$. The orbit of $s_0$ under $H_0$ is $S_0$, for if $gs_0\in S_0$ is an arbitrary element, then there exists $h \in H$ such that $hs_0=gs_0 \in S_0$, which implies $h \in H_0$. Simplicity follows from the simplicity of the $H$-action on $S$.
\item
Suppose $g,g'\in G_0$ and $g\vert_{S_0}=g'\vert_{S_0}$. Then $gs_0=g's_0$ and $g=g'$ by the simple transitivity of the $G$-action on $S$. The injectivity of the $H_0$-embedding is similar.
\item
Suppose $g \in \text{\rm Sym}(S_0)$ and $g$ commutes with $H_0\vert_{S_0}$. By the simple transitivity of the $G$-action on $S$, there is a unique $\overline{g} \in G$ such that $\overline{g}s_0=gs_0$. By \ref{thm:sub_dual_groups:H_0_simple_transitivity} a general element of $S_0$ has the form $hs_0$ with $h \in H_0$. We have $\overline{g}hs_0=h\overline{g}s_0=h\vert_{S_0}gs_0=gh\vert_{S_0}s_0=ghs_0$. Thus $\overline{g}$ and $g$ agree on $S_0$ and $\overline{g}$ is a true extension of $g$. The proof for the analogous statement concerning elements in $\text{\rm Sym}(S_0)$ which commute with $G_0\vert_{S_0}$ is similar.
\item
We already have simply transitive group actions of $G_0\vert_{S_0}$ and $H_0\vert_{S_0}$ on $S_0$ by \ref{thm:sub_dual_groups:G_0_simple_transitivity} and \ref{thm:sub_dual_groups:H_0_simple_transitivity}. The groups $G_0\vert_{S_0}$ and $H_0\vert_{S_0}$ commute because $G$ and $H$ do. Suppose $h \in \text{Sym}(S_0)$ commutes with $G_0\vert_{S_0}$. Then by \ref{thm:sub_dual_groups:extensions}, the bijection $h$ admits an extension to $\overline{h}\in H_0 \subseteq \text{Sym}(S)$, which says precisely $h \in H_0\vert_{S_0}$. This proves  $C_{\text{Sym}(S_0)}(G_0\vert_{S_0})=H_0\vert_{S_0}$. The proof $C_{\text{Sym}(S_0)}(H_0\vert_{S_0})=G_0\vert_{S_0}$ is similar.
\end{enumerate}
\qed
\end{proof}

\begin{myremark}
In Theorem~\ref{thm:sub_dual_groups}, the subgroups $G_0$ and $H_0$ of $\text{Sym}(S)$ commute, since they are subgroups of the dual groups $G$ and $H$, respectively. However, $G_0$ and $H_0$ are {\it not} dual groups in $\text{Sym}(S)$ when $G_0 \lneqq G$, though the restrictions $G_0\vert_{S_0}$ and $H_0\vert_{S_0}$ are dual in $\text{\rm Sym}(S_0)$. First of all, $G_0$ and $H_0$ do not act simply transitively on $S$. Secondly, the centralizer of $G_0$ is larger than $H_0$, as the centralizer contains at least all of $H$. The isomorphism type of $C_{\text{Sym}(S)}(G_0)$ may be found using the methods of
\cite{peckGeneralizedCommuting}.
\end{myremark}

\begin{corollary}[Transforming the orbits of $G_0$] \label{cor:transforming_orbits}
Let $G,H \leq \text{\rm Sym}(S)$ be dual groups, $G_0$ a subgroup of $G$, and $s_0$ an element in $S$. Let $S_0$ be the orbit of $s_0$ under the action of $G_0$, and $H_0$ the subgroup of $H$ consisting of those elements $h\in H$ with $hs_0 \in S_0$. Then the following hold.
\begin{enumerate}
\item \label{cor:transforming_orbits:transform_of_S_0}
If $k\in H$, then the $G_0$-orbit of $ks_0$ is $kS_0$. Thus, if we transform $s_0$ to $ks_0$ in Theorem~\ref{thm:sub_dual_groups}, then $S_0$ transforms to $kS_0$.
\item
If $k\in H$, then $kH_0k^{-1}$ is precisely the subgroup of $H$ consisting of those $h\in H$ such that $h(ks_0) \in kS_0$. Thus, if we transform $s_0$ to $ks_0$ in Theorem~\ref{thm:sub_dual_groups}, then $H_0$ transforms to $kH_0k^{-1}$.
\item
If $k\in H$, then the restrictions of $G_0$ and $kH_0k^{-1}$ to $kS_0$ are dual in $\text{\rm Sym}(kS_0)$.
\item
Each orbit of $G_0$ can be written as $kS_0$ for some $k\in H$. Moreover, the group $H$ acts transitively on the orbits of $G_0$ by $h(kS_0)=(hk)S_0$.
\end{enumerate}
\end{corollary}
\begin{proof}
\begin{enumerate}
\item
The element $k\in H$, commutes with each element in $G_0$, so $G_0ks_0=kG_0s_0=kS_0$.
\item
If $kh_0k^{-1}\in kH_0k^{-1}$, then $kh_0k^{-1}(ks_0)=kh_0s_0\in k S_0$ since $h_0 s_0 \in S_0$. On the other hand, if $h \in H$ and $hks_0=gks_0$ for some $g \in G_0$, then $k^{-1}h k s_0 = gs_0$ and $k^{-1} h k \in H_0$, so $h \in k H_0 k^{-1}$.
\item
This follows directly from Theorem~\ref{thm:sub_dual_groups}~\ref{thm:sub_dual_groups:restricted_subgroups_dual}.
\item
This follows directly from the transitivity of the $H$-action on $S$ and \ref{cor:transforming_orbits:transform_of_S_0}.
\end{enumerate}
\qed
\end{proof}

\subsection{Hexatonic Systems} \label{subsec:Hexatonic_Systems}
To see the utility of Theorem~\ref{thm:sub_dual_groups}, we construct the hexatonic systems of Cohn~\cite{cohn1996} and the dual pair of Clampitt~\cite{clampittParsifal}. In Theorem~\ref{thm:sub_dual_groups}, we take $S$ to be the set of major and minor triads, $G$ to be the neo-Riemannian $PLR$-group, $H$ to be the $T/I$-group, and $G_0$ to be the $PL$-group, which is generated by $P$ and $L$, is isomorphic to $S_3$, and therefore has only 6 elements.

Let $s_0$ be $E\flat$ (we write major chords with upper-case letters and minor chords with lower-case letters). Beginning with $E\flat$, and evaluating $P$ and $L$ alternately, we obtain the 6 chords $\{E\flat, e\flat, B, b, G,g\}$. We have now found the complete orbit $S_0$ of $E\flat$ and may stop: by Theorem~\ref{thm:sub_dual_groups}~\ref{thm:sub_dual_groups:G_0_simple_transitivity}, $G_0$ acts simply transitively on $S_0$, which means $|S_0|=|G_0|=6$ by the Orbit-Stabilizer Theorem.

To find $H_0$, we must only find those transpositions and inversions which map $E\flat$ into $S_0$, or in other words, we must only solve the six equations
$$\aligned
T_iE\flat &=E\flat &\hspace{1in} I_\ell E\flat &=e\flat \\
T_jE\flat &=G &  I_mE\flat&=g \\
T_kE\flat&=B& I_nE\flat&=b.\endaligned$$
Clearly, $i$, $j$, and $k$ are $0$, $4$, and $8$ respectively, and less clearly, $\ell$, $m$, and $n$ are 1, 5, and 9. In cycle notation, the group $G_0|_{S_0}$ is
\begin{center}
$G_0|_{S_0}:$ \hspace{.25in}
\begin{tabular}{ll}
$\text{Id}=()$ & $LP|_{S_0}=(E\flat\; G \;B)(e\flat \; b \; g)$\\
$P|_{S_0}=(E\flat\;e\flat)(G\;g)(B\;b)$ \;\;\;\;\;\; & $PL|_{S_0}=(E\flat\;B\;G)(e\flat\;g\;b)$\\
$L|_{S_0}=(E\flat\;g)(G\; b)(B \;e\flat)$ \;\;\;\;\;\; & $PLP|_{S_0}=(E\flat \; b)(G \; e\flat)(B\;g)$
\end{tabular}
\end{center}
and the group $H_0|_{S_0}$ is
\begin{center}
$H_0|_{S_0}:$ \hspace{.25in}
\begin{tabular}{ll}
$T_0|_{S_0}=()$ & $I_1|_{S_0}=(E\flat\; e\flat)(G\; b)(B \; g)$\\
$T_4|_{S_0}=(E\flat\;G\;B)(e\flat\;g\;b)$ \;\;\;\;\;\; & $I_5|_{S_0}=(B\;b)(G\;e\flat)(E\flat\;g)$\\
$T_8|_{S_0}=(E\flat\;B\;G)(e\flat\;b\;g)$ \;\;\;\;\;\; & $I_9|_{S_0}=(G \; g)(E\flat\;b)(B\;e\flat),$
\end{tabular}
\end{center}
exactly as in Clampitt's article \cite{clampittParsifal}.
The groups $G_0|_{S_0}$ and $H_0|_{S_0}$ are dual in $\text{Sym}(S_0)$ by Theorem~\ref{thm:sub_dual_groups}~\ref{thm:sub_dual_groups:restricted_subgroups_dual}.

We may now use Corollary~\ref{cor:transforming_orbits} and the above computation to find the other three hexatonic systems and the corresponding duals to the restricted $PL$-groups. In Corollary~\ref{cor:transforming_orbits}, we take $k=T_1$, $T_2$, and $T_3$. Starting with $T_1$, we find the $PL$-orbit
$$T_1S_0=\{E,e,C,c,A\flat,a\flat\}$$
and (extended) dual group
$$T_1H_0T_1^{-1}=\{T_0,T_4,T_8,I_3,I_7,I_{11}\}.$$
With $T_2$ we obtain the $PL$-orbit
$$T_2S_0=\{F,f,C\sharp,c\sharp,A,a\}$$
and (extended) dual group
$$T_2H_0T_2^{-1}=\{T_0,T_4,T_8,I_5,I_9,I_{1}\}.$$
With $T_3$ we obtain the $PL$-orbit
$$T_3S_0=\{G\flat,f\flat,D,d,B\flat,b\flat\}$$
and (extended) dual group
$$T_2H_0T_2^{-1}=\{T_0,T_4,T_8,I_7,I_{11},I_{3}\}.$$

We know we have all orbits and may now stop because the union of the orbits thus far gives us all of $S$. For a direct computation of all four hexatonic systems, see the exposition \cite{oshita}.

\subsection{Octatonic Systems} \label{subsec:Octatonic_Systems}
As another application of Theorem~\ref{thm:sub_dual_groups} and Corollary~\ref{cor:transforming_orbits}, we construct the three octatonic systems. In Theorem~\ref{thm:sub_dual_groups}, we take $S$ to be the set of major and minor triads, $G$ to be the neo-Riemannian $PLR$-group, $H$ to be the $T/I$-group, and $G_0$ to be the $PR$-group, which is generated by $P$ and $R$, is isomorphic to $D_4$, and therefore has only 8 elements.\footnote{To see that the $PR$-group is dihedral of order 8, we note that alternately applying $P$ and $R$ to $C$ gives $C,c,E\flat,e\flat,G\flat,g\flat,A,a,C$, so that $s:=RP$ has order 4 (recall $P$ and $R$ commute with $T_1$). The element $t:=P$ has order 2, and the equation $tst=s^{-1}$ holds true. Thus the $PR$-group is generated by $\{s,t\}$ and $s$ and $t$ satisfy the relations of $D_4$.}

Let $s_0$ be the $C$ chord. Beginning with $C$, and evaluating $P$ and $R$ alternately, we obtain the 8 chords $\{C,c,E\flat,e\flat,G\flat,g\flat,A,a\}$. We have now found the complete orbit $S_0$ of $C$ and may stop: by Theorem~\ref{thm:sub_dual_groups}~\ref{thm:sub_dual_groups:G_0_simple_transitivity}, $G_0$ acts simply transitively on $S_0$, which means $|S_0|=|G_0|=8$ by the Orbit-Stabilizer Theorem.

To find $H_0$, we must only find those transpositions and inversions which map $C$ into $S_0$, or in other words, we must only solve the eight equations
$$\aligned
T_iC&=C & \hspace{1in} I_m C&=c \\
T_jC&=E\flat &   I_nC&=e\flat \\
T_kC&=G\flat &   I_oC&=g\flat \\
T_{\ell}C&=A &   I_pC&=a.
\endaligned$$
Clearly, $i$, $j$, $k$, and $\ell$ are $0$, $3$, $6$, and $9$ respectively, and with a little more work, $m$, $n$, $o$, and $p$ are 7, 10, 1, and 4.
The groups $G_0|_{S_0}$ and $H_0|_{S_0}$ are dual in $\text{Sym}(S_0)$ by Theorem~\ref{thm:sub_dual_groups}~\ref{thm:sub_dual_groups:restricted_subgroups_dual}.

We may now use Corollary~\ref{cor:transforming_orbits} and the above computation to find the other two octatonic systems and the corresponding duals to the restricted $PR$-groups. In Corollary~\ref{cor:transforming_orbits}, we take $k=T_1$ and $T_2$. Starting with $T_1$, we find the $PR$-orbit
$$T_1S_0=\{D\flat,d\flat,E,e,G,g,B\flat,b\flat\}$$
and (extended) dual group
$$T_1H_0T_1^{-1}=\{T_0,T_3,T_6,T_9,I_9,I_0,I_3,I_6\}.$$
With $T_2$ we obtain the $PR$-orbit
$$T_2S_0=\{D,d,F,f,A\flat,a\flat,B,b\}$$
and (extended) dual group
$$T_2H_0T_2^{-1}=\{T_0,T_3,T_6,T_9,I_{11},I_2,I_{5},I_8\}.$$

We know we have all orbits and may now stop because the union of the orbits thus far gives us all of $S$.  For a direct computation of all three octatonic systems, see \cite{oshita}.

\section{The Topos of Triads and the neo-Riemannian $PLR$-Group} \label{sec:topos_of_triads}

We first recall the topos of triads $\mathbf{Set}^{\mathcal{T}}$ from \cite{nollToposOfTriads} and then present the enumeration. In the present article we use the semi-tone encoding of pitch classes, and not the circle-of-fifths encoding used in \cite{nollToposOfTriads}. Consequently, some formulas will differ, though the objects under consideration are exactly the same.

The {\it triadic monoid} $\mathcal{T}$ of \cite{nollToposOfTriads} is the collection of affine maps $\mathbb{Z}_{12} \to \mathbb{Z}_{12}$ which preserve the $C$-triad $\{0,4,7\}$ as a set. A map $\psi \colon\mathbb{Z}_{12} \to \mathbb{Z}_{12}$ is {\it affine} if it has the form $z \mapsto mz+b$ for some $m, b \in \mathbb{Z}_{12}$. There are 144 affine maps $\mathbb{Z}_{12} \to \mathbb{Z}_{12}$, the invertible ones are characterized by $m \in \{1,5,7,11\}$. A function $\psi$ {\it preserves $\{0,4,7\}$ as a set} if $\psi(\{0,4,7\}) \subseteq \{0,4,7\}$. This is weaker than preserving $\{0,4,7\}$ {\it pointwise}. The set $\mathcal{T}$ forms a {\it monoid} under function composition, that is, function composition is an associative, unital operation on $\mathcal{T}$ (functions in $\mathcal{T}$ are not necessarily invertible). The triadic monoid $\mathcal{T}$ has 8 elements and is generated by the affine maps $f(z):=3z+7$ and $g(z):=8z+4$. The other maps in $\mathcal{T}$ are $f\circ f(z)=9z+4$, $g\circ g(z)=4z$,  the three constant maps 0, 4, 7, and the identity $\text{Id}_{\mathbb{Z}_{12}}$.

If $\mathcal{S}\subseteq \mathbb{Z}_{12}$ is closed under the action of $\mathcal{T}$, that is $\mathcal{T}\mathcal{S} \subset \mathcal{S}$, then $\mu\vert_\mathcal{S}\colon \mathcal{T} \times \mathcal{S} \to \mathcal{S}$ is called a {\it subaction of $\mu$} and $\mathcal{S}$ is called the {\it carrier set of the subaction $\mu|_\mathcal{S}$}.

The {\it topos of triads} is the category $\mathbf{Set}^{\mathcal{T}}$, its objects are sets equipped with a $\mathcal{T}$-action and its morphisms are $\mathcal{T}$-equivariant maps. Like any presheaf category, this category is a {\it topos}. See \cite{nollToposOfTriads} for a complete investigation of the triadic monoid and the topos of triads. An exposition of \cite{nollToposOfTriads} may be found in \cite{bartlett}. See the Appendix of the present paper for a recollection of topos-theoretic concepts used in this article.

Our aim in the present article is to clarify the relationship between the neo-Riemannian $PLR$-group, the triadic monoid, and the topos of triads, so we focus on subactions of the natural action $\mu\colon \mathcal{T} \times \mathbb{Z}_{12} \to \mathbb{Z}_{12}$ which are covered by consonant triads. A tremendous difference between the triadic monoid action and the $PLR$-group action is that the triadic monoid acts on pitch classes, whereas the $PLR$-group acts on aggregate chords; the $PLR$-group simply {\it does not act on individual pitch classes}. However, we may use Theorem~\ref{thm:sub_dual_groups} to induce an action of certain transpositions and inversions on the pitch classes of a monoid subaction in the case that the carrier set is covered by an orbit of a $PLR$-subgroup.

For example, consider the major-minor mixture $\{0,3,4,7\}$. A quick computation verifies that $\mathcal{T} \{0,3,4,7\} \subseteq \{0,3,4,7\}$, as it suffices to check if the two generators $f$ and $g$ of $\mathcal{T}$ preserve $\{0,3,4,7\}$. The carrier set $\{0,3,4,7\}$ of this monoid subaction is covered by the consonant triads $C$ and $c$, while $\{C,c\}$ is an orbit of the $PLR$-subgroup $\{\text{Id},P\}$. Theorem~\ref{thm:sub_dual_groups} now implies that the dual subgroup in $\text{Sym}(\{C,c\})$ is $\{\text{Id},I_7\}$. The identity and the inversion operator $I_7$ {\it do act on individual pitch classes}, and since $\text{Id}$ and $I_7$ preserve the set $\{C,c\}$, they will also preserve the set of underlying pitch classes. This is how we move from $PLR$-subgroup actions on chords to $T/I$-subgroup actions on pitch classes. It is a coincidence in this very small example that $I_7$ and $P$ agree on the set $\{C,c\}$.

We may now enumerate all $\mathcal{T}$-sets $\mathcal{S}$ which are covered by triads and admit an action of a $PLR$-subgroup on the maximal cover $\mathbb{S}$, see Figure~\ref{figure:enumeration} and Figure~\ref{figure:carrier_sets}. In Figure~\ref{figure:enumeration}, the operation $Q_6$ transposes major chords up by 6 semitones and minor chords down by 6 semitones (this is the same as $T_6$, but we call it $Q_6$ to emphasize that we think of it as element of the $PLR$-group). The notation S$\ell$ indicates the {\it slide transformation} on major and minor chords, which holds the third constant and moves the root and the fifth in parallel by a semitone.

\begin{theorem}[Enumeration] \label{thm:triadic_monoid_and_PLR_subgroups}
Let $\mathcal{T}$ be the triadic monoid, that is, $\mathcal{T}$ is the monoid of affine functions $\mathbb{Z}_{12} \to \mathbb{Z}_{12}$ which preserve the $C$-chord $\{0,4,7\}$ as a set. Suppose $\mathcal{S} \subseteq \mathbb{Z}_{12}$ is closed under the natural $\mathcal{T}$-action in the sense that $\mathcal{T}\mathcal{S} \subseteq \mathcal{S}$. Assume further that $\mathcal{S}$ is covered by consonant triads, and let $\mathbb{S}$ be the collection of consonant triads $Y$ with $Y \subseteq \mathcal{S}$. If $\mathbb{S}$ admits a simply transitive action of a $PLR$-subgroup $G$, then $\mathcal{S}$ must be one of the sets enumerated in Figure~\ref{figure:enumeration}.
\end{theorem}
\begin{figure}
\begin{center}
\begin{tabular}{|l|c|c|c|}
\hline
Carrier Set $\mathcal{S}$ & Type & Maximal Covering $\mathbb{S}$ & $PLR$-Subgroup \\ \hline
$\{0,4,7\}$ & Major Chord & $C$ & $\{\text{Id}\}$ \\ \hline
$\{0,3,4,7\}$ & Major-Minor Mixture & $C,c$ & $\{\text{Id},P\}$ \\ \hline
$\{0,3,4,7,8,11\}$ & Hexatonic & $C,c,A\flat,a\flat,E,e$ & $\langle P,L \rangle$ \\ \hline
$\{0,1,3,4,6,7,9,10\}$ & Octatonic & $C,c,E\flat,e\flat,G\flat,g\flat,A,a$ & $\langle P,R \rangle$ \\ \hline
$\{0,1,4,6,7,10\}$ & Major Triad Tritone Mixture & $C,G\flat$ & $\{\text{Id},Q_6\}$ \\ \hline
$\{0,1,2,4,6,7,8,10\}$ & Prometheus Tritone Mixture & $C,d\flat,G\flat,g$ & $\{\text{Id},Q_6,\text{S}\ell,Q_6\text{S}\ell\}$ \\ \hline
$\mathbb{Z}_{12}$ & Chromatic Scale & All Consonant Triads & $PLR$-group \\ \hline
\end{tabular}
\end{center}
\caption{Enumeration of all $\mathcal{T}$-Subactions with Maximal Cover Admitting a Simply Transitive Action of a $PLR$-Subgroup} \label{figure:enumeration}
\end{figure}

\begin{figure}
\begin{center}
\begin{tabular}{|l|l|}
\hline
Carrier Set $\mathcal{S}$ in $\mathbb{Z}_{12}$ & Carrier Set $\mathcal{S}$ in Musical Notation \\ \hline
$\{0,4,7\}$ & $\{C,E,G\}$ \\ \hline
$\{0,3,4,7\}$ & $\{C,E\flat,E,G\}$ \\ \hline
$\{0,3,4,7,8,11\}$  & $\{C,E\flat,E,G,A\flat,B\}$ \\ \hline
$\{0,1,3,4,6,7,9,10\}$ & $\{C,D\flat,E\flat,E,G\flat,G,A,B\flat\}$ \\ \hline
$\{0,1,4,6,7,10\}$ & $\{C,D\flat,E,G\flat,G,B\flat\}$ \\ \hline
$\{0,1,2,4,6,7,8,10\}$ & $\{C,D\flat,D,E,G\flat,G,A\flat,B\flat\}$ \\ \hline
$\mathbb{Z}_{12}$ & $\{C,D\flat,D,E\flat,E,F,G\flat,G,A\flat,A,B\flat,B\}$ \\ \hline
\end{tabular}
\end{center}
\caption{Carrier Sets of Figure~\ref{figure:enumeration}} \label{figure:carrier_sets}
\end{figure}
\begin{proof}
Suppose $G$ is a $PLR$-subgroup which acts simply transitively on $\mathbb{S}$.

\noindent {\bf Case 1: $G$ contains $P$.} Assume $P \in G$. Our case-by-case method is to consider the orbit of $C=\{0,4,7\}$ under each $PLR$-subgroup which contains $P$, and to test if the underlying pitch-class set is closed under $\mathcal{T}$. By Theorem~\ref{thm:sub_dual_groups}~\ref{thm:sub_dual_groups:G_0_simple_transitivity}, each $PLR$-subgroup acts simply transitively on its $C$-orbit.

We may quickly determine all $PLR$-subgroups containing $P$ via the explicit structure of the $PLR$-group, described for instance in~\cite{cransfioresatyendra} and \cite{fioresatyendra2005}. The $PLR$-group has two generators, $Q_1$ of order 12 and $P$ of order 2, where $Q_1$ transposes major chords up by $1$ and minor chords down by $1$. The bijection $Q_k$ is defined similarly, and satisfies $Q_k=(Q_1)^{k}$. The 24 elements of the $PLR$-group are $\{Q_k|k \in \mathbb{Z}_{12}\}\cup \{PQ_k|k \in \mathbb{Z}_{12}\}$, for example $L=PQ_4$ and $R=PQ_9$. For these reasons, any $PLR$-subgroup containing $P$ is generated by $P$ and some $Q_i$. The group $\{Q_k|k \in \mathbb{Z}_{12}\}$ is cyclic, hence all of its subgroups are cyclic, and for our computation of the $PLR$-subgroups containing $P$ it suffices to check $\langle P, Q_i \rangle$ where $Q_i$ is a generator of a cyclic subgroup, that is, $Q_i \in \{\text{Id},Q_1,Q_2,Q_3,Q_4,Q_6\}$.

The smallest $PLR$-subgroup containing $P$ is $\{\text{Id},P\}$, which produces the $C$-orbit $\{C,c\}$ with underlying pitch-class set $\{0,3,4,7\}$. This is closed under the $\mathcal{T}$-action, as $f(3)=4=g(3)$.

The next larger subgroup is $\langle P,Q_6\rangle=\{\text{Id},Q_6,P,PQ_6\}$, which produces the $C$-orbit $\{C,c, G\flat,g\flat\}$ with underlying pitch-class set $\{0,1,3,4,6,7,9,10\}$, the octatonic. The group $\langle P,Q_6\rangle$ does not act simply transitively on the set of all consonant triads in the octatonic, so a hypothesis of the theorem is not satisfied, and there is nothing to check.

We arrive next at $\langle P, Q_4 \rangle=\{\text{Id},Q_4,Q_8,P,PQ_4,PQ_8\}$, which is the same as $\langle P,PQ_4\rangle=\langle P,L\rangle$. The $C$-orbit is a hexatonic system of Cohn $\{C,c,A\flat,a\flat,E,e\}$, see~\cite{cohn1996} and Section~\ref{subsec:Hexatonic_Systems} of the present paper. The underlying pitch-class set $\{0,3,4,7,8,11\}$ is closed under $\mathcal{T}$, since $f(8)=7$, $f(11)=4$, $g(8)=8$, and $g(11)=8$.

The group $\langle P, Q_3 \rangle=\{\text{Id},Q_3,Q_6,Q_9,P,PQ_3,PQ_6,PQ_9\}$ is the same as $\langle P, PQ_9\rangle=\langle P, R\rangle$. The $C$-orbit is an octatonic system $$\{C,c,E\flat,e\flat,G\flat,g\flat,A,a\},$$ see Section~\ref{subsec:Octatonic_Systems}. The underlying pitch-class set $\{0,1,3,4,6,7,9,10\}$ is closed under $\mathcal{T}$, since $f(1)=10$, $f(6)=1$, $f(9)=10$, $f(10)=1$, and $g(1)=0$, $g(6)=4$, $g(9)=4$, $g(10)=0$.

The group $\langle P, Q_2\rangle$ has $\mathbb{Z}_{12}$ as the underlying pitch-class set for its $C$-orbit because $\{0,3,4,7\}$ has a representative from each coset $\mathbb{Z}_{12}/6\mathbb{Z}_{12}$. That is, iterated applications of $Q_2$ to $\{0,4,7\}$ and $\{0,3,7\}$ will reach all of $\mathbb{Z}_{12}$. However, the group $\langle P,Q_2 \rangle$ does not act simply transitively on the set of all consonant triads in $\mathbb{Z}_{12}$, so a hypothesis of the theorem is not satisfied, and there is nothing to check.

The group $\langle P, Q_1 \rangle$ is the entire $PLR$-group, and its $C$-orbit is the entire collection of major and minor triads. Its underlying pitch-class set is of course $\mathbb{Z}_{12}$, which is clearly closed under the $\mathcal{T}$-action.

We have considered all $PLR$-subgroups containing $P$ and determined which ones produce a $C$-orbit with underlying pitch-class set closed under the $\mathcal{T}$-action.

\noindent {\bf Case 2: $G$ does not contain $P$.} Assume $P \notin G$ and $G$ acts simply transitively on $\mathbb{S}$.

If $G$ is trivial, then the $C$-orbit is simply $\{C\}$ with underlying pitch-class set $\{0,4,7\}$, which is closed under the $\mathcal{T}$-action by definition of $\mathcal{T}$.

For the rest of Case 2, we assume $G$ is nontrivial and argue using the sub group $H \leq T/I$ which maps $C=\{0,4,7\}$ to triads in $\mathbb{S}$. The sub group $H$ was called $H_0$ in Theorem~\ref{thm:sub_dual_groups}, and is defined via the action on consonant triads. We will use the  $H$-action on individual pitch classes to conclude that the only two possibilities for $\mathcal{S}$ are the Major Triad Tritone Mixture $\{0,1,4,6,7,10\}$ and the Prometheus Tritone Mixture $\{0,1,2,4,6,7,8,10\}$ whenever $P \notin G$.

We first remark $3 \notin \mathcal{S}$, for if $3$ were in $\mathcal{S}$, then $\{0,3,7\}$ and $\{0,4,7\}$ would be in $\mathbb{S}$, which would necessitate $P \in G$. Based on this observation, we may conclude that the only transposition in $H$ is $T_6$, as follows. Since $3 \notin \mathcal{S}$, we have $T_3 \notin H$, and consequently $T_9 \notin H$. Similarly, $T_8 \notin H$, for if $T_8$ were in $H$, then $T_8\{0,4,7\}=\{8,0,3\}$ would be a subset of $\mathcal{S}$, a contradiction. Consequently, $T_2,T_4,T_{10} \notin H$. Clearly, $T_1,T_5,T_7,T_{11} \notin H$, for if any of these were in $H$, all transpositions would be in $H$. Thus, the only transposition that can be in $H$ is $T_6$. The transposition $T_6$ must be in $H$, since $H \leq T/I$ is nontrivial by Theorem~\ref{thm:sub_dual_groups}~\ref{thm:sub_dual_groups:H_0_simple_transitivity}, the composite of two inversions is a transposition, and the only allowable transposition is $T_6$.

At this point, we note that $\{\text{Id},T_6\}$ acts simply transitively on $\{C,G\flat\}$, and its dual group $\{\text{Id},Q_6\}$, determined by Theorem~\ref{thm:sub_dual_groups}, is a possibility for $G$. The underlying pitch-class set is the Major Triad Tritone Mixture $\{0,1,4,6,7,10\}$, which is closed under the $\mathcal{T}$-action because $f(1)=10$, $f(6)=1$, $f(10)=1$, and $g(1)=0$, $g(6)=4$, $g(10)=0$, so the Major Triad Tritone Mixture is a possibility for $\mathcal{S}$.

Let us return to general $G$ with $P \notin G$. We claim $9 \notin \mathcal{S}$. Since $T_6 \in H$ and $H$ preserves $\mathbb{S}$, so also $\mathcal{S}$, we have $T_6\{0,4,7\}=\{6,10,1\}=G\flat \subseteq \mathcal{S}$. If $9$ were in $\mathcal{S}$, then $g\flat=\{6,9,1\}$ would be a subset of $\mathcal{S}$, which would necessitate $P \in G$, a contradiction. Thus $9 \notin \mathcal{S}$.

We also claim $5 \notin \mathcal{S}$. If $5$ were in $\mathcal{S}$, then $f\circ f(5)=1$ and $g(5)=8$ must also belong to $\mathcal{S}$. Then we have the parallel triads
$\{1,4,8\}$ and $\{1,5,8\}$ both in $\mathbb{S}$, which is not allowed by $P \notin G$. Hence, $5 \notin \mathcal{S}$.

Suppose for the remainder of the proof that $H$ properly contains $\{\text{Id},T_6\}$. Our next goal is to show that the other elements of $H$ are exactly $I_2$ and $I_8$. The other elements of $H$ must be inversions, as the only transposition in $H$ is $T_6$ as we proved above. Recall that the inversion $I_{k+\ell}$ is the unique inversion which interchanges $k$ and $\ell$. Since $3,5,9 \notin \mathcal{S}$ as we saw above, we can exclude from $H$ all inversions that interchange 0, 4, or 7 with 3, 5, or 9. Thus, $I_3,I_7,I_{10},I_{5},I_9,I_0,I_9,I_1,I_4 \notin H$. The composite of $T_6$ with any of these cannot be in $H$, so we also have $I_{11}, I_6\notin H$. The only two inversions that can be in $H$ are $I_2$ and $I_8$. But if either is in $H$, then so is the other, so both $I_2$ and $I_8$ must be in $H$.  Thus $H=\{\text{Id},T_6,I_2,I_8\}$.

At this point, we note that $\{\text{Id},T_6,I_2,I_8\}$ acts simply transitively on the chord set $\{C,d\flat,G\flat,g\}$ and the dual subgroup is $\{\text{Id},Q_6,\text{S}\ell,Q_6\text{S}\ell\}$ by Theorem~\ref{thm:sub_dual_groups}. The underlying pitch-class set $\{0,1,2,4,6,7,8,10\}$ is closed under the $\mathcal{T}$ action, since $f(1)=10$, $f(2)=1$, $f(6)=1$, $f(8)=7$, $f(10)=1$, and $g(1)=0$, $g(2)=8$, $g(6)=4$, $g(8)=8$, $g(10)=0$. We conclude that the Prometheus Tritone Mixture $\{0,1,2,4,6,7,8,10\}$ is the final possibility for $\mathcal{S}$. \qed
\end{proof}

We may now consider upgrades by the Lawvere--Tierney topologies $$j_\mathcal{P},j_\mathcal{L},j_\mathcal{R} \colon \Omega \to \Omega$$ in the topos $\mathbf{Set}^\mathcal{T}$, found on page 117 of \cite{nollToposOfTriads}. Recall that if $\nu$ is a sub $\mathcal{T}$-action of a $\mathcal{T}$-action $\sigma$ and its characteristic morphism is $\chi\colon \sigma \to \Omega$, then the {\it $j$-upgrade} of $\nu$ has carrier set equal to the pre-image of the point $\mathcal{T} \in \Omega$ under $j \circ \chi$, that is, the $j$-upgrade has carrier set $(j \circ \chi)^{-1}(\mathcal{T})$.

\begin{corollary}[Maximal covers of Lawvere--Tierney upgrades admit $PLR$-subgroup actions]
Let $\varphi\colon \mathbb{Z}_{12} \to \mathbb{Z}_{12}$ be an element of the $T/I$-group and let $(\mathbb{Z}_{12},\mu^\varphi)$ denote the $\mathcal{T}$-action given by the $\varphi$-conjugation of the natural $\mathcal{T}$-action, that is, the generators $f$ and $g$ of $\mathcal{T}$ act by $\varphi f \varphi^{-1}$ and $\varphi g \varphi^{-1}$. Then the maximal covers of the $j_\mathcal{P}$-, $j_\mathcal{L}$-, and $j_\mathcal{R}$-upgrades of the subaction $(\varphi(C),\mu^{\varphi}|_{\varphi(C)})$ admit simply transitive actions by $PLR$-subgroups. In fact, the upgrades and their maximal covers are the $\varphi$-images of the upgrades of $C$ and their respective covers. The $PLR$-subgroups remain the same. See Figure~\ref{figure:upgrades}.
\end{corollary}
\begin{figure}
\begin{center}
\begin{tabular}{|l|l|c|c|}
\hline
Upgrade & Carrier Set  & Type  & $PLR$-Subgroup \\ \hline
$j_\mathcal{P}$ & $\varphi\{0,3,4,7\}$ & Major-Minor Mixture &  $\langle P \rangle$ \\ \hline
$j_\mathcal{L}$ & $\varphi\{0,3,4,7,8,11\}$ & Hexatonic &  $\langle P,L \rangle$ \\ \hline
$j_\mathcal{R}$ & $\varphi\{0,1,3,4,6,7,9,10\}$ & Octatonic &  $\langle P,R \rangle$ \\ \hline
\end{tabular}
\end{center}
\caption{Upgrades of the Subaction $(\varphi(C),\mu^{\varphi}|_{\varphi(C)})$ by the Lawvere--Tierney Topologies $j_\mathcal{P}$, $j_\mathcal{L}$, and $j_\mathcal{R}$} \label{figure:upgrades}
\end{figure}
\begin{proof}
Let $\chi^C\colon \mathbb{Z}_{12} \to \Omega$ denote the characteristic morphism of the subaction $\{0,4,7\}$ of the natural $\mathcal{T}$-action on $\mathbb{Z}_{12}$. Then we have the following two pullbacks (see the Appendix in Section~\ref{sec:appendix} for the notion of pullback).
$$\xymatrix@C=3pc@R=3pc{\varphi(C) \ar[r]^-{\varphi^{-1}} \ar@{^{(}->}[d] \ar@{}[dr]|{\text{pullback}} & C \ar[r] \ar@{^{(}->}[d] \ar@{}[dr]|{\text{pullback}} & \{\mathcal{T}\} \ar@{^{(}->}[d] \\ \mathbb{Z}_{12} \ar[r]_-{\varphi^{-1}} & \mathbb{Z}_{12} \ar[r]_{\chi^C} & \Omega }$$
The composite of these two pullback squares is also a pullback, so the characteristic morphism of $(\varphi(C),\mu^{\varphi}|_{\varphi(C)})$ is $\chi^C \circ \varphi^{-1}$. Thus, the $j$-upgrade of $(\varphi(C),\mu^{\varphi}|_{\varphi(C)})$ is $\varphi$ of the $j$-upgrade of $(C,\mu^{\text{Id}})$, that is,
$$(j \circ \chi^{\varphi(C)})^{-1}(\mathcal{T})=(j \circ \chi^C \circ \varphi^{-1})^{-1}(\mathcal{T})=\varphi\left((j \circ \chi^C)^{-1}(\mathcal{T}) \right).$$
We recall from page 123 of \cite{nollToposOfTriads} that the $j_\mathcal{P}$-, $j_\mathcal{L}$-, and $j_\mathcal{R}$-upgrades of $C$ are the major-minor mixture, hexatonic, and octatonic appearing in Figure~\ref{figure:enumeration}. The maximal cover of the $\varphi$-image is clearly the $\varphi$-image of the maximal cover, so we are precisely in the situation of Corollary~\ref{cor:transforming_orbits}. Sections~\ref{subsec:Hexatonic_Systems}~and~\ref{subsec:Octatonic_Systems} further illustrate the point. \qed
\end{proof}

In future work, we will study triadic coverings of sets which are simultaneously invariant under multiple conjugated $\mathcal{T}$-actions. Another goal is to better understand the role of the three groups $\langle P \rangle$, $\langle P,L \rangle$, $\langle P,R \rangle$  within the framework of other groups which act simply transitively on triadic coverings.

\section{Appendix on Topos Theory} \label{sec:appendix}

We recall some basic notions from topos theory. For explanations, examples, and proofs we refer to the excellent textbook of Mac Lane--Moerdijk \cite{maclanemoerdijk}. Only a very small fraction of this Appendix is actually used in the present paper, so readers do not need to understand this Appendix in order to follow the present paper. A topos is a special kind of category, so we begin with categories.

A {\it category} $\mathbf{C}$ consists of a class $\text{Obj}\; \mathbf{C}$ of {\it objects}, a set $\text{Mor}_\mathbf{C}(A,B)$ of {\it morphisms from $A$ to $B$} for each pair of objects $A$ and $B$, and a {\it composition}
$$\xymatrix{\text{Mor}_\mathbf{C}(B,C) \times \text{Mor}_\mathbf{C}(A,B) \ar[r]^-{\circ} & \text{Mor}_\mathbf{C}(A,C)}$$
for each triple of objects $A$, $B$, and $C$. The composition is required to be {\it associative} and {\it unital}.
For instance, the category $\mathbf{Set}$ has sets as objects and functions as morphisms. An prominent example from Section~\ref{sec:topos_of_triads} is the category $\mathbf{Set}^\mathcal{T}$, in which an object is a set equipped with an action of the triadic monoid $\mathcal{T}$ and a morphism is a $\mathcal{T}$-equivariant function.
Groups and group homomorphisms form a third example of a category.

A {\it topos} is a category $\mathbf{C}$ such that the following three axioms hold.
\begin{enumerate}
\item
The category $\mathbf{C}$ admits all {\it finite limits}.
\item
The category $\mathbf{C}$ admits a {\it subobject classifier} $\ast \to \Omega$. This is an object $\Omega$
equipped with a universal monomorphism from the terminal object $\ast$ to $\Omega$, which is universal in the sense that
any other monomorphism is a pullback of it in a unique way. Often, one simply calls $\Omega$ the {\it subobject classifier}.
\item  \label{Cartesian_closedness}
The category $\mathbf{C}$ is {\it Cartesian closed}, that is, for any two objects $B$ and $C$ there is an object $C^B$ in $\mathbf{C}$ and for any object $A$ of $\mathbf{C}$ a bijection
$$\text{Mor}_\mathbf{C}(A \times B,C) \cong \text{Mor}_\mathbf{C}(A,C^B)$$
natural in $A$, $B$, and $C$.
\end{enumerate}
To understand these axioms, we consider how the category $\mathbf{Set}$ is a topos. This category admits all finite limits. For instance, the diagram $$\xymatrix{ & B \ar[d]^g \\  A \ar[r]_f & C}$$ has finitely many objects and morphisms, and $L:=\{(a,b) \in A \times B \;\vert \; f(a)=g(b)\}$ makes the diagram
\begin{equation} \label{equ:pullback_defn}
\xymatrix{L \ar[d]_{\text{pr}_1} \ar[r]^{\text{pr}_2} & B \ar[d]^g \\  A \ar[r]_f & C}
\end{equation}
commute, and moreover there is a unique morphism from $L$ to any analogous $L'$ in such a way the two triangles involving $L$, $L'$, $A$, and $B$ also commute. This $L$ is called a {\it limit} of the diagram, and the limit for this special diagram is called a {\it pullback}. The projection $\text{pr}_1$ is called a {\it pullback} of $g$ and the projection $\text{pr}_2$ is called a {\it pullback} of $f$.

For the discussion of the subobject classifier of $\mathbf{Set}$, we first point out that the monomorphisms in $\mathbf{Set}$ are precisely the injective maps. The subobject classifier in $\mathbf{Set}$ is the inclusion $\{1\} \hookrightarrow \{0,1\}$. If $D$ is a subset of $E$, then the inclusion $D \hookrightarrow E$ is a pullback of $\{1\} \hookrightarrow \{0,1\}$ in a unique way: there is only one function $\chi$ that makes the following diagram a pullback.
$$\xymatrix{ D \ar@{^{(}->}[d] \ar[r] &  \{1\} \ar@{^{(}->}[d] \\ E \ar[r]_{\chi} & \{0,1\} }$$
This function $\chi$ is necessarily the usual {\it characteristic function}, which takes the value 1 for inputs in $D$ and the value 0 for inputs in $E \backslash D$. Note that $\chi^{-1}(1)=D$.

For the Cartesian closedness of $\mathbf{Set}$, the object $C^B$ is defined to be the set of functions $B \to C$, that is, $C^B=\text{Mor}_{\mathbf{Set}}(B,C)$. The bijection in \ref{Cartesian_closedness} sends a function $f\colon A \times B \to C$ to the function $A \to C^B$ defined by $a \mapsto \left( b \mapsto f(a,b) \right)$. Thus, $\mathbf{Set}$ is a topos.

Many other examples of topoi are given by presheaf categories. A {\it presheaf} on a category $\mathbf{D}$ is a functor from the opposite category of $\mathbf{D}$ to $\mathbf{Set}$ (here $\mathbf{D}$ is not required to be a topos). For any category $\mathbf{D}$, the category of presheaves on $\mathbf{D}$ is a topos. In particular, if $G$ is a group or a monoid, then the category of $G$-sets is a topos (when $G$ is viewed as a category with one object, a presheaf on $G^{\text{op}}$ is the same as a set equipped with a $G$-action, and a morphism of presheaves is a $G$-equivariant map). In particular, for the triadic monoid $\mathcal{T}$, the category $\mathbf{Set}^\mathcal{T}$ of $\mathcal{T}$-sets and their morphisms is a topos.

In the topos $\mathbf{Set}^\mathcal{T}$, the underlying set of the subobject classifier $\Omega$ is the set of left ideals in $\mathcal{T}$. For instance, the full monoid $\mathcal{T}$ and the empty set $\emptyset$ are both left ideals in $\mathcal{T}$, so $\mathcal{T}, \emptyset \in \Omega$.
To obtain the other left ideals it is convenient to inspect a Cayley graph of $\mathcal{T}$ with respect to its generators $f$ and $g$.

\hspace{1 cm}

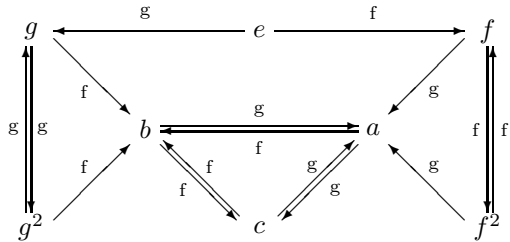
\begin{center}

\unitlength 1 cm
\begin{picture}(6,3)
\put(3,3){\makebox(0,0){$e$}}
\put(3,0.4){\makebox(0,0){$c$}}
\put(0,3){\makebox(0,0){$g$}}
\put(6,3){\makebox(0,0){$f$}}
\put(0,2.8){\vector(0,-1){2.2}}
\put(6,2.8){\vector(0,-1){2.2}}

\put(-0.07,0.6){\vector(0,1){2.2}}
\put(6.07,0.6){\vector(0,1){2.2}}

\put(1.5,1.7){\makebox(0,0){$b$}}
\put(4.5,1.7){\makebox(0,0){$a$}}
\put(5.7,2.9){\vector(-1,-1){1}}
\put(0.3,2.9){\vector(1,-1){1}}
\put(4.3,1.5){\vector(-1,-1){1}}
\put(1.7,1.5){\vector(1,-1){1}}
\put(5.7,0.5){\vector(-1,1){1}}
\put(0.3,0.5){\vector(1,1){1}}

\put(3.25,0.55){\vector(1,1){1}}
\put(2.74,0.55){\vector(-1,1){1}}

\put(0,0.4){\makebox(0,0){$g^{2}$}}
\put(6,0.4){\makebox(0,0){$f^{2}$}}
\put(3.2,3){\vector(1,0){2.5}}
\put(2.8,3){\vector(-1,0){2.5}}
\put(1.7,1.74){\vector(1,0){2.6}}
\put(4.3,1.67){\vector(-1,0){2.6}}

\scriptsize
\put(3, 1.95){\makebox(0,0){$\mathrm{g}$}}
\put(3, 1.45){\makebox(0,0){$\mathrm{f}$}}
\put(4.5, 3.25){\makebox(0,0){$\mathrm{f}$}}
\put(1.5, 3.25){\makebox(0,0){$\mathrm{g}$}}

\put(5.85, 1.7){\makebox(0,0){$\mathrm{f}$}}
\put(0.15, 1.7){\makebox(0,0){$\mathrm{g}$}}
\put(6.23, 1.7){\makebox(0,0){$\mathrm{f}$}}
\put(-0.23, 1.7){\makebox(0,0){$\mathrm{g}$}}

\put(5.3, 2.2){\makebox(0,0){$\mathrm{g}$}}
\put(0.7, 2.2){\makebox(0,0){$\mathrm{f}$}}
\put(5.3, 1.2){\makebox(0,0){$\mathrm{g}$}}
\put(0.7, 1.2){\makebox(0,0){$\mathrm{f}$}}
\put(4.0, 0.9){\makebox(0,0){$\mathrm{g}$}}
\put(2.35, 1.2){\makebox(0,0){$\mathrm{f}$}}
\put(3.70, 1.2){\makebox(0,0){$\mathrm{g}$}}
\put(2, 0.9){\makebox(0,0){$\mathrm{f}$}}
\normalsize
\end{picture}
\end{center}
The nodes labelled $e, f, f^2, g, g^2, a, b, c$ represent the $8$ elements of $\mathcal{T}$. The left ideals are node sets where no arrow starts that leads outside of this set. One may directly check by eye that there are four non-empty proper subsets $\mathcal{B} \subset \mathcal{T}$ satisfying this condition, namely $$\mathcal{C} = \{a, b, c\}, \, \,  \mathcal{L} = \{a, b, c, f, f^2\},\, \,  \mathcal{R} = \{a, b, c, g, g^2\}, \, \,  \mathcal{P} = \{a, b, c, f, f, ^2, g, g^2\}.$$ The $\mathcal{T}$-action on the subobject classifier $\Omega=\{\emptyset, \mathcal{C}, \mathcal{L}, \mathcal{R}, \mathcal{P}, \mathcal{T}\}$ is
$$m \cdot \mathcal{B}:=\{n \in \mathcal{T} \; \vert\; n \cdot m \in \mathcal{B}\}$$
for $m \in \mathcal{T}$ and $\mathcal{B} \in \Omega$,
and the universal monomorphism is $\{T\} \hookrightarrow\Omega$.

A {\it Lawvere--Tierney topology} in the topos $\mathbf{Set}^\mathcal{T}$ is a $\mathcal{T}$-equivariant map $\Omega \to \Omega$ such that
\begin{enumerate}
\item
$j(\mathcal{T})=\mathcal{T}$,
\item
$j \circ j = j$,
\item
$j(\frak{r}) \cap j(\frak{s})=j(\frak{r} \cap \frak{s})$ for all $\mathfrak{r}, \mathfrak{s} \in \Omega$.
\end{enumerate}
The six Lawvere--Tierney topologies of $\mathbf{Set}^\mathcal{T}$ are determined in Sections 2.3 and 2.4 of \cite{nollToposOfTriads}.
Given a sub $\mathcal{T}$-set $D$ of a $\mathcal{T}$-set $E$, one may {\it upgrade} it by a Lawvere--Tierney topology $j$ in the following way. By definition of subobject classifier, the inclusion $D \hookrightarrow E$ corresponds uniquely to a morphism $\chi \colon E \rightarrow \Omega$, its characteristic morphism. The {\it $j$-upgrade of $D$} is then the sub $\mathcal{T}$-set of $E$ which has characteristic morphism $j \circ \chi$. The underlying set of the upgrade is $(j \circ \chi)^{-1}(\mathcal{T})$.

For example, the characteristic morphism of the $C$-major chord $\{0,4,7\}$ in $\mathbb{Z}_{12}$ with the natural $\mathcal{T}$-action is
\begin{center}
\begin{tabular}{|c|c|c|c|c|c|c|c|c|c|c|c|c|}
\hline
$t$ & 0 & 1 & 2 & 3 & 4 & 5 & 6 & 7 & 8 & 9 & 10 & 11 \\ \hline
$\chi^{\{0,4,7\}}(t)$ & $\mathcal{T}$ & $\mathcal{R}$ & $\mathcal{C}$ & $\mathcal{P}$ & $\mathcal{T}$ & $\mathcal{C}$ & $\mathcal{R}$ & $\mathcal{T}$ & $\mathcal{L}$ & $\mathcal{R}$ & $\mathcal{R}$ & $\mathcal{L}$ \\ \hline
\end{tabular}
\end{center}
as can be determined by precomposing the characteristic morphism for $\chi^{\{0,1,4\}}$ on page 121 of \cite{nollToposOfTriads} with multiplication by 7. Using this, and the tables on pages 117 and 118 of \cite{nollToposOfTriads}, the upgrades of $\{0,4,7\}$ are found to be as follows.
\begin{center}
\begin{tabular}{|l|l|c|}
\hline
Upgrade & Carrier Set  & Type \\ \hline
$j_\mathcal{T}$ & $\{0,4,7\}$ & Major Chord \\ \hline
$j_\mathcal{P}$ & $\{0,3,4,7\}$ & Major-Minor Mixture  \\ \hline
$j_\mathcal{L}$ & $\{0,3,4,7,8,11\}$ & Hexatonic \\ \hline
$j_\mathcal{R}$ & $\{0,1,3,4,6,7,9,10\}$ & Octatonic \\ \hline
$j_\mathcal{C}$ & $\mathbb{Z}_{12}$ & Chromatic Scale \\ \hline
$j_\mathcal{F}$ & $\mathbb{Z}_{12}$ & Chromatic Scale \\ \hline
\end{tabular}
\end{center}
The numbers in this table appear to differ from those on page 123 in \cite{nollToposOfTriads} because we are using the semitone encoding in the present article, rather than the circle-of-fifths encoding. The two tables are related by multiplication with 7.

A surprising discovery of \cite{nollToposOfTriads} is that these music-theoretic collections are the Lawvere--Tierney upgrades of the $C$-major chord. One of the goals of the present paper is to thicken this music-theoretic meaning by way of the $PLR$-group in Theorem~\ref{thm:triadic_monoid_and_PLR_subgroups}.

\subsubsection*{Acknowledgements.}
Thomas M. Fiore thanks Ramon Satyendra for explaining David Clampitt's article \cite{clampittParsifal} to him in 2004.


\end{document}